\documentclass[a4paper,12pt]{amsart}
\usepackage{verbatim}
\usepackage[english]{babel}
\usepackage[utf8]{inputenc}
\usepackage{amsthm}
\usepackage[T1]{fontenc}
\usepackage{amsmath}
\usepackage{enumerate}
\usepackage{mathabx}
\usepackage{mathtools}
\usepackage{amssymb}
\usepackage{tikz}
\theoremstyle{plain}
\newtheorem{Teo}{Theorem}[section]

\theoremstyle{definition}
\newtheorem{Ex}[Teo]{Example}
\newtheorem{Def}[Teo]{Definition}
\numberwithin{equation}{section}

\DeclareMathOperator{\Lip}{Lip}

\DeclareMathOperator{\spanop}{span}

\title{Characterisation of the weak-star symmetric strong diameter 2 property in Lipschitz spaces}
\author{Andre Ostrak}
\date{}

\begin{document}

\begin{abstract}
We give a characterisation of the weak$^*$ symmetric strong diameter $2$ property for Lipschitz function spaces in terms of a property of the corresponding metric space. Using this characterisation we show that the weak$^*$ symmetric strong diameter $2$ property is different from the weak$^*$ strong diameter $2$ property in Lipschitz spaces, thereby answering a question posed in a recent paper by Haller, Langemets, Lima, and Nadel.
\end{abstract}

\maketitle

\section{Introduction}
We consider only real Banach spaces. We start by fixing some notation. Given a metric space $M$ and a point $x$ in $M$, we denote by $B(x,r)$ the open ball in $M$ centered at $x$ of radius $r$. Let $X$ be a Banach space. We denote the closed unit ball, the unit sphere, and the dual space of $X$ by $B_X$, $S_X$, and $X^*$, respectively. A \emph{weak$^*$ slice} of $B_{X^*}$ is a set of the form
\[
S(B_{X^*},x, \alpha)\coloneqq \{x^*\in B_{X^*}\colon \langle x, x^*\rangle>1-\alpha\},
\]
where $x\in S_X$ and $\alpha>0$.

Let $M$ be a pointed metric space, that is, a metric space with a fixed point $0$. The space $\Lip_0(M)$ of all Lipschitz functions $f\colon M\to \mathbb{R}$ with $f(0)=0$ is a Banach space with the norm
\[
\|f\|_{\Lip} =\sup\left\{ \frac{|f(x)-f(y)|}{d(x,y)}\colon x,y\in M, x\neq y\right\}.
\]
The space
\[
\mathcal{F}(M)\coloneqq \overline{\spanop}\left\{\delta_m \colon m\in M \right\}\subset \Lip_0(M)^*
\]
is called the Lipschitz-free space over $M$, where $\delta_m\colon \Lip_0(M)\to \mathbb{R}$,
\[
\langle f, \delta_m\rangle=f(m),\qquad m\in M,\; f\in \Lip_0(M).
\]
It can be shown that, under this duality, $\mathcal{F}(M)^*$ is isometrically isomorphic to $\Lip_0(M)$.

Recall that the dual space $X^*$ is said to have the \emph{weak$^*$ strong diameter $2$ property} ($w^*$-SD$2$P) if every finite convex combination of weak$^*$ slices of $B_{X^*}$ has diameter $2$. It is well known that $X^*$ has the $w^*$-SD$2$P iff the norm of $X$ is octahedral (\cite{D},\cite{G}, for a proof, see, e.g., \cite{HLP}). Therefore, $\Lip_0(M)$ has the $w^*$-SD$2$P iff the norm of $\mathcal{F}(M)$ is octahedral. Moreover, in \cite[Theorem 3.1]{PR}, it was shown that the norm of $\mathcal{F}(M)$ is octahedral iff the metric space $M$ has the following property.

\begin{Def}
A metric space $M$ is said to have the \emph{long trapezoid property} (LTP) if, for every finite subset $N$ of $M$ and $\varepsilon>0$, there exist $u,v\in M$, $u\neq v$, such that, for any $x,y\in N$,
\begin{equation}\label{eq: LTP}
(1-\varepsilon)\bigl(d(x,y)+d(u,v)\bigr)\leq d(x,u)+d(y,v).
\end{equation}
\end{Def}

\noindent
Therefore, the Lipschitz space $\Lip_0(M)$ has the $w^*$-SD$2$P iff $M$ has the LTP. The objective of this paper is to give a similar characterisation to the following property, which was introduced in \cite{ALN} but studied more extensively in \cite{ANP}, \cite{HLLN}, \cite{CCGMR}, and \cite{LR}.
\begin{Def}
A dual Banach space $X^*$ is said to have the \emph{weak$^*$ symmetric strong diameter $2$ property} ($w^*$-SSD$2$P) if, for every finite family $\{S_i\}_{i=1}^n$ of weak$^*$ slices of $B_{X^*}$ and $\varepsilon>0$, there exist $f_i\in S_i$, $i=1,\ldots,n$, and $g\in B_{X^*}$ such that $f_i\pm g\in S_i$ for every $i\in \{1,\ldots,n\}$ and $\|g\|>1-\varepsilon$.
\end{Def}

\noindent
It is known that in general the $w^*$-SSD$2$P is a strictly stronger property than the $w^*$-SD$2$P (see, e.g., \cite{HLLN}). In this paper, we show that the same is true for Lipschitz function spaces, thus giving an answer to \cite[Question 6.3]{HLLN}.

The paper is organised as follows.

In Section \ref{sec: 2}, we give a characterisation of the $w^*$-SSD$2$P for the Lipschitz space $\Lip_0(M)$ in terms of a property of the metric space $M$. More precisely, we prove Theorem \ref{main}, which says that $\Lip_0(M)$ has the $w^*$-SSD$2$P iff $M$ enjoys the following property.

\begin{Def}
We say that $M$ has the \emph{strong long trapezoid property} (SLTP) if, for every finite subset $N$ of $M$ and $\varepsilon>0$, there exist $u,v\in M$,  $u\neq v$, such that, for any $x,y \in N$, the inequality \eqref{eq: LTP} holds, and, for any $x,y,z,w\in N$,

\begin{equation}\label{eq: SLTP}
\begin{aligned}
(1-\varepsilon)&\bigl(2d(u,v)+d(x,y)+d(z,w)\bigr)\\
&\qquad\qquad\leq d(x,u)+d(y,u)+d(z,v)+d(w,v).
\end{aligned}
\end{equation}
\end{Def}

In Section \ref{sec: 3}, we first apply Theorem \ref{main} to show that, for Lipschitz spaces, the $w^*$-SSD$2$P is a strictly stronger property than the $w^*$-SD$2$P: Example \ref{ex1} provides a metric space which has the LTP but not the SLTP.


A question that arises from the definition of the SLTP is whether the inequality \eqref{eq: SLTP} implies \eqref{eq: LTP}. Example \ref{ex2} shows that this is not the case: it provides a metric space $M$ for which \eqref{eq: SLTP} holds for every finite subset $N$, but which fails the LTP.

We finish the paper by showing that any infinite subset of $\ell_1$, viewed as a metric space, has the SLTP (Example \ref{ell1}).

\section{Main result}\label{sec: 2}

\begin{Teo}\label{main}
Let $M$ be a pointed metric space. The following statements are equivalent:
\begin{enumerate}[\upshape (i)]
    \item $\Lip_0(M)$ has the $w^*$-SSD$2$P;
    \item $M$ has the SLTP.
\end{enumerate}
\end{Teo}

\begin{proof}
(i)$\Rightarrow$(ii).
Assume that $\Lip_0(M)$ has the $w^*$-SSD$2$P,
and let $N$ be a finite subset of $M$ and $0<\varepsilon<1$.
Choose $\alpha>0$ such that $2\alpha<\varepsilon$ and, for any $x,y\in N$, $x\neq y$,
\[
\alpha\leq\frac{1}{d(x,y)}
\quad\text{and}\quad
2\alpha\leq d(x,y).
\]
For any $x,y\in N$, $x\neq y$,
define a slice $S_{x,y}:=S\left(B_{\Lip_0(M)}, \frac{\delta_x-\delta_y}{d(x,y)}, \alpha^3\right)$.
Since $\Lip_0(M)$ has the $w^*$-SSD$2$P, we can find $f_{x,y}\in S_{x,y}$ and $g\in B_{\Lip_0(M)}$, $\|g\|\geq 1-\alpha$, such that $\|f_{x,y}\pm g\|\leq 1$. For $x,y\in N$, $x=y$, define $f_{x,y}:=0\in\Lip_0(M)$.

For any $x,y\in N$,
\begin{align*}
\langle f_{x,y}, \delta_x-\delta_y \rangle
=f_{x,y}(x)-f_{x,y}(y)
\geq(1-\alpha^3)d(x,y),
\end{align*}
therefore, keeping in mind that $\|f_{x,y}\pm g\|\leq1$,
\[
|\langle g, \delta_x-\delta_y \rangle|=|g(x)-g(y)|\leq \alpha^3 d(x,y)\leq \alpha^2.
\]
Since $\|g\|\geq 1-\alpha$, there exist $u,v\in M$, $u\neq v$, such that
\[
\langle g, \delta_u-\delta_v \rangle
=g(u)-g(v)\geq\left(1-\alpha\right)d(u,v).
\]
Now, for any $x,y\in N$, again using that $\|f_{x,y}\pm g\|\leq1$,
\begin{align*}
|\langle f_{x,y},\delta_u-\delta_v\rangle|=|f_{x,y}(u)-f_{x,y}(v)|\leq \alpha d(u,v).
\end{align*}

Letting $x,y,z,w\in N$ be arbitrary, it remains to verify \eqref{eq: LTP} and \eqref{eq: SLTP}.
Since $\|f_{x,y}\pm g\|\leq 1$, we get
\begin{align*}
(1-\varepsilon)&\bigl(d(u,v)+d(x,y)\bigr)\\
&\leq (1-2\alpha)d(u,v)+(1-2\alpha^3)d(x,y)\\
&\leq \langle g,\delta_u-\delta_v\rangle - \langle f_{x,y},\delta_u-\delta_v\rangle
+\langle f_{x,y}, \delta_x-\delta_y\rangle - \langle g, \delta_x-\delta_y\rangle\\
&=\langle g-f_{x,y},\delta_u-\delta_x\rangle-\langle g-f_{x,y}, \delta_v-\delta_y\rangle\\
&\leq d(x,u)+d(y,v).
\end{align*}
Thus, \eqref{eq: LTP} holds.
If $x=y$ and $z=w$, then \eqref{eq: SLTP} follows from \eqref{eq: LTP} with $y$ replaced by $z$.
If $x\neq y$ or $z\neq w$, then
\[
\alpha\bigl(d(x,y)+d(z,w)\bigr)
\geq 2\alpha^2
\geq|\langle g, \delta_z-\delta_x+\delta_w-\delta_y \rangle|,
\]
and thus, since $\|f_{x,y}\pm g\|\leq 1$,
\begin{align*}
(1-\varepsilon)&\bigl(2d(u,v)+d(x,y)+d(z,w)\bigr)\\
&\leq 2\bigl(g(u)-g(v)\bigr)+(1-\alpha^3-\alpha)\bigl(d(x,y)+d(z,w)\bigr)\\
&\leq 2\langle g, \delta_u-\delta_v \rangle + \langle f_{x,y}, \delta_x-\delta_y \rangle
      +\langle f_{z,w}, \delta_z-\delta_w \rangle\\
&\qquad +\langle g, \delta_z-\delta_x+\delta_w-\delta_y \rangle\\
&=\langle g-f_{x,y}, \delta_u-\delta_x \rangle + \langle g+f_{x,y}, \delta_u-\delta_y \rangle\\
&\qquad -\langle g+f_{z,w}, \delta_v-\delta_z \rangle - \langle g-f_{z,w}, \delta_v-\delta_w \rangle\\
&\leq d(x,u)+d(y,u)+d(z,v)+d(w,v).
\end{align*}

\medskip
(ii)$\Rightarrow$(i).
Assume that $M$ has the SLTP. Let $n\in\mathbb{N}$, let $S_i:=S(B_{\Lip_0(M)}, \mu_i,\alpha_i)$, $i=1,\dotsc,n$,
be weak$^\ast$ slices of $B_{\Lip_0(M)}$, and let $0<\varepsilon<1$.
It suffices to find $f_i\in S_i$, $i=1,\dotsc,n$, and $g\in B_{\Lip_0(M)}$ with $\|g\|\geq(1-\varepsilon)^2$ such that
$\|f_i\pm g\|\leq1$ for every $i\in\{1,\dotsc,n\}$.
We may assume that, for every $i\in\{1,\dotsc,n\}$, one has $\mu_i=\sum_{j=1}^{n_i}\lambda_{ij}\delta_{x_{ij}}$
for some $n_i\in\mathbb{N}$, $\lambda_{ij}\in\mathbb{R}\setminus\{0\}$, and $x_{ij}\in M$, $j=1,\dotsc,n_i$.
Now $N:=\{0\}\cup\bigcup_{i=1}^{n}\{x_{i1},\dotsc,x_{in_i}\}$ is a finite subset of $M$.
We may also assume that $\varepsilon<\min_{1\leq i\leq n}\alpha_i$.
This enables, for every $i\in\{1,\dotsc,n\}$, to pick an $ h_i\in S_i$ with $\| h_i\|<1-\varepsilon$.

By the SLTP, there exist $u,v\in M$, $u\neq v$, satisfying \eqref{eq: LTP} and \eqref{eq: SLTP} for all $x,y,z,w\in N$.
Setting
\[
r_0:=\frac{1}{2}\min_{x,y\in N}\bigl(d(x,u)+d(y,u)-(1-\varepsilon)d(x,y)\bigr)
\]
and
\[
s_0:=\frac{1}{2}\min_{z,w\in N}\bigl(d(z,v)+d(w,v)-(1-\varepsilon)d(z,w)\bigr),
\]
one has $r_0+s_0\geq(1-\varepsilon)d(u,v)$. Thus, there exist $r,s\geq 0$ with $r \leq  r_0$ and $s\leq s_0$ such that
\[
r+s=(1-\varepsilon)^2 d(u,v).
\]
We may assume that $r>0$. Define a function $g\colon M\to\mathbb{R}$ by
\[
g(x):=
\begin{cases}
r-d(x,u)&\quad\text{if $x\in B(u,r)$;}\\
-s+d(x,v)&\quad\text{if $x\in B(v,s)$;}\\
0&\quad\text{otherwise}
\end{cases}
\]
(we use the convention $B(v,s)=\emptyset$ if $s=0$). Observe that $\|g\|\leq 1$
(here we use that, whenever $x\in B(u,r)$ and $y\in B(v,s)$, one has $g(y)\leq 0 \leq g(x)$, and thus $|g(x)-g(y)|=g(x)-g(y)$).
One also has $\|g\|\geq(1-\varepsilon)^2$, because
\[
|g(u)-g(v)|=g(u)-g(v)=r+s=(1-\varepsilon)^2 d(u,v).
\]

Set $L:=N\cup B$ where $B:=B(u,r)\cup B(v,s)$.
We next show that, for every $i\in\{1,\dotsc,n\}$, there is a $c_i\in\mathbb{R}$ such that,
defining a function $f_i\colon L\to\mathbb{R}$ by $f_i|_N= h_i|_N$ and $f_i|_{B}=c_i$ (observe that $B\cap N=\emptyset$), one has
$\|f_i\pm g\|_{\Lip_0(L)}\leq 1$ and $\|f_i\pm |g|\|_{\Lip_0(L)}\leq 1$.

Let $i\in\{1,\dotsc,n\}$. Set
\begin{alignat*}{2}
\widecheck{a}_i&:=\max_{x\in N}\bigl( h_i(x)-d(x,u)\bigr),\quad
&\widehat{a}_i&:=\min_{x\in N}\bigl( h_i(x)+d(x,u)\bigr),\\
\widecheck{b}_i&:=\max_{x\in N}\bigl( h_i(x)-d(x,v)\bigr),\quad
&\widehat{b}_i&:=\min_{x\in N}\bigl( h_i(x)+d(x,v)\bigr).
\end{alignat*}
Whenever $x,y\in N$, since $\| h_i\|<1-\varepsilon$, one has
\[
 h_i(x)+d(x,u)-\bigl( h_i(y)-d(y,u)\bigr)\geq d(x,u)+d(y,u)-(1-\varepsilon)d(x,y)\geq 2r,
\]
and, by \eqref{eq: LTP},
\begin{align*}
 h_i(x)+d(x,u)-\bigl( h_i(y)-d(y,v)\bigr)
&\geq d(x,u)+d(y,v)-(1-\varepsilon)d(x,y)\\
&\geq (1-\varepsilon)d(u,v)> r+s.
\end{align*}
Thus, $\widehat{a}_i-r\geq \widecheck{a}_i+r$ and $\widehat{a}_i-r> \widecheck{b}_i+s$.
Similarly, one observes that $\widehat{b}_i-s\geq \widecheck{b}_i+s$
and $\widehat{b}_i-s > \widecheck{a}_i+r$.
It follows that there exists a
$c_i\in\bigl[\widecheck{a}_i+r,\widehat{a}_i-r\bigr]\cap\bigl[\widecheck{b}_i+s,\widehat{b}_i-s\bigr]$.
This $c_i$ does the job.

Indeed, let $x\in N$ and $y\in B(u,r)$.
In order to see that
\[
\bigl|f_i(x)\pm g(x)-\bigl(f_i(y)\pm g(y)\bigr)\bigr|
=\Bigl| h_i(x)-\Bigl(c_i\pm \bigl(r-d(y,u)\bigr)\Bigr)\Bigr|
\leq d(x,y),
\]
it suffices to show that
\begin{equation}\label{eq: ...=<C_i-+A=<...}
 h_i(x)-d(x,y)\pm d(y,u)
\leq c_i\pm r
\leq h_i(x)+d(x,y)\pm d(y,u).
\end{equation}
These inequalities hold:
\begin{align*}
 h_i(x)-d(x,y)-d(y,u)
&\leq h_i(x)-d(x,u)\\
&\leq\widecheck{a}_i\leq c_i-r\leq\widehat{a}_i-2r\\
&\leq h_i(x)+d(x,u)-2d(y,u)\\
&\leq h_i(x)+d(x,y)-d(y,u)
\end{align*}
and
\begin{align*}
 h_i(x)-d(x,y)+d(y,u)
&\leq h_i(x)-d(x,u)+2d(y,u)\\
&\leq\widecheck{a}_i+2r\leq c_i+r\leq\widehat{a}_i\\
&\leq h_i(x)+d(x,u)\\
&\leq h_i(x)+d(x,y)+d(y,u).
\end{align*}
The inequalities
\[
\bigl|f_i(x)\pm|g(x)|-\bigl(f_i(y)\pm |g(y)|\bigr)\bigr|
=\Bigl| h_i(x)-\Bigl(c_i\pm \bigl(r-d(y,u)\bigr)\Bigr)\Bigr|
\leq d(x,y)
\]
follow from \eqref{eq: ...=<C_i-+A=<...}.

For every $i\in\{1,\dotsc,n\}$, we extend $f_i$ to the entire space $M$ by setting
\[
f_i(y):=\sup_{x\in L}\bigl(f_i(x)+|g(x)|-d(x,y)\bigr)
\quad\text{for every $y\in M\setminus L$.}
\]
Note that, on $M\setminus L$, the function $f_i$ agrees with a norm preserving extension of $(f_i+|g|)|_L$.
%
%
%
It remains to show that $\|f_i\pm g\|_{\Lip_0(M)}\leq1$.
Indeed, this implies that also  $\|f_i\|_{\Lip_0(M)}\leq 1$, and thus $f_i\in S_i$,
because, since $f_i|_N= h_i|_N$, one has $\langle\mu_i, f_i\rangle= \langle \mu_i, h_i\rangle>1-\alpha_i$. 

Let $i\in\{1,\dotsc,n\}$.
To see that $\|f_i\pm g\|_{\Lip_0(M)}\leq1$, it suffices to show that,
whenever $x,y\in M$, one has
\begin{equation}\label{eq: -d(x,y)=<f_i(x)+-g(x)-(f_i(y)+-g(y))=<d(x,y)}
-d(x,y)\leq f_i(x)\pm g(x)-\bigl(f_i(y)\pm g(y)\bigr)\leq d(x,y).
\end{equation}
For the cases when $x,y\in L$ or $x,y\in M\setminus L$,
or $x\in N$ (or $y\in N$) and $y\in M\setminus L$ (or $x\in M\setminus L$),
the inequalities \eqref{eq: -d(x,y)=<f_i(x)+-g(x)-(f_i(y)+-g(y))=<d(x,y)}
follow from what has been proven above.
So, in fact, it suffices to consider the case when $x\in B(u,r)\cup B(v,s)$ and $y\in M\setminus L$.
In this case, \eqref{eq: -d(x,y)=<f_i(x)+-g(x)-(f_i(y)+-g(y))=<d(x,y)} means that
\[
-d(x,y)\leq c_i\pm g(x)-\sup_{z\in L}\bigl(f_i(z)+|g(z)|-d(z,y)\bigr)\leq d(x,y).
\]
Thus, it suffices to show that
\begin{enumerate}
\item
there is a $z\in L$ such that
\[
c_i\pm g(x)-d(x,y)+d(z,y)\leq f_i(z)+|g(z)|;
\]
\item
for every $z\in L$,
\[
f_i(z)+|g(z)|\leq c_i\pm g(x)+d(x,y)+d(z,y).
\]
\end{enumerate}
For (1), one may take $z=x$, so it remains to prove (2).
By symmetry, it suffices to consider only the case when $x\in B(u,r)$.
In this case $g(x)=r-d(x,u)\geq0$. Thus, it suffices to prove that, for every $z\in L$,
\begin{equation*}\label{eq: what one must show when x in B(u,A)}
f_i(z)+|g(z)|\leq c_i-r +d(x,u)+d(x,y)+d(z,y).
\end{equation*}
One has to look through the following cases:
\[
(\text{a})\quad z\in B(u,r);\qquad\qquad
(\text{b})\quad z\in B(v,s);\qquad\qquad
(\text{c})\quad z\in N.
\]

(a).
If $z\in B(u,r)$, then $f_i(z)=c_i$ and $|g(z)|=r-d(z,u)$.
Thus, one has to show that
\[
2r\leq d(x,u)+d(z,u)+d(x,y)+d(z,y).
\]
This inequality holds, because, since $y\notin B(u,r)$, one has $d(y,u)\geq r$, and thus
\[
2r\leq d(y,u)+d(y,u)\leq d(x,u)+d(x,y)+d(z,u)+d(z,y).
\]

(b).
If $z\in B(v,s)$, then  $f_i(z)=c_i$ and $|g(z)|=s-d(z,v)$.
Thus, one has to show that
\[
r+s\leq d(x,u)+d(z,v)+d(x,y)+d(z,y).
\]
This inequality holds, because, since $y\notin B(u,r)$ and $y\notin B(v,s)$,
one has $d(y,u)\geq r$ and $d(y,v)\geq s$, and thus
\[
r+s\leq d(y,u)+d(y,v)\leq d(x,u)+d(x,y)+d(z,v)+d(z,y).
\]

(c).
If $z\in N$, then
\begin{align*}
f_i(z)+|g(z)|&=f_i(z)= h_i(z)\leq\widecheck{a}_i+d(z,u)\\
&\leq c_i-r+d(x,u)+d(x,y)+d(z,y).
\end{align*}
%
\end{proof}

\section{Examples}\label{sec: 3}

We now give an example of a metric space $M$ that has the LTP but fails the SLTP. By \cite[Theorem 3.1]{PR} and Theorem \ref{main}, this implies that the corresponding Lipschitz space $\Lip_0(M)$ has the $w^*$-SD$2$P but fails the $w^*$-SSD$2$P.

\begin{Ex}\label{ex1}
Let $M=\{a_1, a_2, b_1, b_2\}\cup \{u_i,v_i\colon i\in \mathbb{N}\}$ be a metric space where the distances between different points are defined as follows: for any $i\in \{1,2\}$, $j,k,l\in \mathbb{N}$, $k\neq l,$
\begin{align*}
d(a_1,a_2)&=d(b_1,b_2)=d(a_i,v_j)=d(b_i,u_j)\\
&=d(u_k,u_l)=d(v_k,v_l)=d(u_k,v_l)=2
\end{align*}
and, for any $i,j\in \{1,2\}$, $k\in \mathbb{N}$,
\begin{align*}
d(a_i,b_j)=d(a_i,u_k)=d(b_i,v_k)=d(u_k,v_k)=1.
\end{align*}

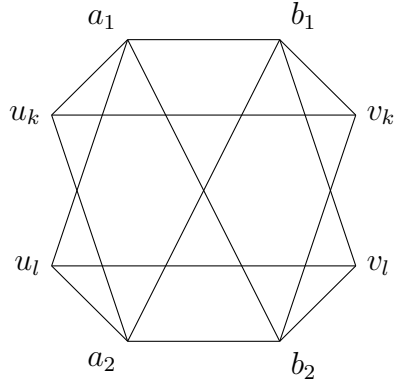
\begin{figure}[h]
    \centering
\begin{tikzpicture}
\draw (1,4) node[above left] {$a_1$};
\draw (1,0) node[below left] {$a_2$};
\draw (3,4) node[above right] {$b_1$};
\draw (3,0) node[below right] {$b_2$};
\draw (0,3) node[left] {$u_k$};
\draw (0,1) node[left] {$u_l$};
\draw (4,3) node[right] {$v_k$};
\draw (4,1) node[right] {$v_l$};
\draw (1,4) -- (3,4);
\draw (1,4) -- (3,0);
\draw (1,4) -- (0,3);
\draw (1,4) -- (0,1);
\draw (1,0) -- (3,4);
\draw (1,0) -- (3,0);
\draw (1,0) -- (0,3);
\draw (1,0) -- (0,1);
\draw (3,4) -- (4,3);
\draw (3,4) -- (4,1);
\draw (3,0) -- (4,3);
\draw (3,0) -- (4,1);
\draw (0,3) -- (4,3);
\draw (0,1) -- (4,1);
\end{tikzpicture}
    \caption{A representation of the metric space $M$ in Example \ref{ex1}. The distances between points connected by a straight line segment are $1$, the distances between other different points are $2$.}
    \label{fig1}
\end{figure}

We first show that $M$ has the LTP. Letting $N$ be a finite subset of $M$ and $i\in \mathbb{N}$ be such that $u_i,v_i\in M\setminus N$, it suffices to show that, for any $x,y\in N$,
\[
d(x,y)+d(u_i,v_i)=d(x,y)+1\leq d(x,u_i)+d(y,v_i).
\]
To this end, letting $x,y\in M\setminus \{u_i,v_i\}$ be such that $d(x,y)=2$, it suffices to show that
\[
d(x,u_i)+d(y,v_i)\geq 3.
\]
For this, notice that if $x\in \{a_1, a_2\}$, then either $y\in \{a_1,a_2\}$ or $y\in \{v_j\colon j\in \mathbb{N}\setminus\{i\} \}$, but in both of these cases $d(y,v_i)=2$ and $d(x,u_i)=1$; if $x\in \{b_1,b_2\}\cup \{u_j,v_j\colon j\in \mathbb{N}\setminus\{i\}\}$, then $d(x,u_i)=2$ and $d(y,v_i)\geq 1$.

It remains to show that $M$ fails the SLTP. Take $N\coloneqq \{a_1,a_2,b_1,b_2\}$. Then, for any $u,v\in M$, $u\neq v$, there exist $x,y,z,w\in N$ such that
\[
2d(u,v)+d(x,y)+d(z,w)\geq d(x,u)+d(y,u)+d(z,v)+d(w,v)+1.
\]
Indeed, set $U\coloneqq \{u_i\colon i\in \mathbb{N}\}$ and $V\coloneqq\{v_i\colon i\in \mathbb{N}\}$, and suppose that $u,v\in M$, $u\neq v$.
If $u,v\in U$ or $u,v\in V$, then, respectively, for $x=z=a_1$, $y=w=a_2$, and for $x=z=b_1$, $y=w=b_2$,
\begin{align*}
2d(u,v)+d(x,y)+d(z,w)&=8>4\\
&=d(x,u)+d(y,u)+d(z,v)+d(w,v).
\end{align*}
If $u\in U$ and $v\in V$, or $u\in V$ and $v\in U$, then, respectively, for $x=a_1$, $y=a_2$, $z=b_1$, $w=b_2$, and for $x=b_1$, $y=b_2$, $z=a_1$, $w=a_2$,
\begin{align*}
2d(u,v)+d(x,y)+d(z,w)&\geq 6>4\\
&=d(x,u)+d(y,u)+d(z,v)+d(w,v).
\end{align*}
Finally, if $u\in N$ or $v\in N$, then, respectively, for $x=y=u$ and $z,w\in N$ with $d(z,w)=2$ and $d(z,v)=d(w,v)=1$, and for $z=w=v$ and $x,y\in N$ with $d(x,y)=2$ and $d(x,u)=d(y,u)=1$,
\begin{align*}
2d(u,v)+d(x,y)+d(z,w)&\geq 4>2\\
&=d(x,u)+d(y,u)+d(z,v)+d(w,v).
\end{align*}
\end{Ex}

The following example shows that the inequality \eqref{eq: SLTP} in the definition of the SLTP does not imply \eqref{eq: LTP}.

\begin{Ex}\label{ex2}
Let $M=\{a,b\}\cup \{u_i,v_i\colon i\in \mathbb{N}\}$ be a metric space where the distances between different points are defined as follows: for any $i, j\in \mathbb{N}$, $i\neq j$,
\begin{align*}
d(a,b)=d(a,v_i)=d(b, u_i)=d(u_i,u_j)=d(v_i,v_j)=d(u_i,v_j)=2
\end{align*}
and, for any $i\in \mathbb{N}$,
\begin{align*}
d(a,u_i)=d(b,v_i)=d(u_i,v_i)=1.
\end{align*}

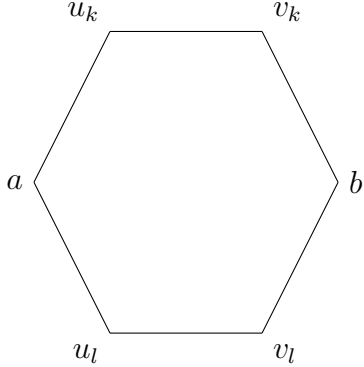
\begin{figure}[h]
    \centering
\begin{tikzpicture}[scale=1]
\draw (0,2) node[left] {$a$};
\draw (4,2) node[right] {$b$};
\draw (1,4) node[above left] {$u_k$};
\draw (1,0) node[below left] {$u_l$};
\draw (3,4) node[above right] {$v_k$};
\draw (3,0) node[below right] {$v_l$};
\draw (0,2) -- (1,4);
\draw (0,2) -- (1,0);
\draw (4,2) -- (3,4);
\draw (4,2) -- (3,0);
\draw (1,4) -- (3,4);
\draw (1,0) -- (3,0);
\end{tikzpicture}
    \caption{A representation of the metric space $M$ in Example \ref{ex2}. The distances between points connected by a straight line segment are $1$, the distances between other different points are $2$.}
    \label{fig2}
\end{figure}

For any finite subset $N$ of $M$, we can find an $i\in \mathbb{N}$ such that $u_i,v_i\in M\setminus N$. We first show that, for any $x,y,z,w\in N$,
\begin{align*}
d(x,y)+d(z,w)+2d(u_i,v_i)&=d(x,y)+d(z,w)+2\\
&\leq d(x,u_i)+d(y,u_i)+d(z,v_i)+d(w,v_i).
\end{align*}
By symmetry it suffices to show that, for any $x,y\in M\setminus \{u_i, v_i\}$,
\[
d(x,y)+1\leq d(x,u_i)+d(y,u_i).
\]
This inequality holds trivially if $d(x,u_i)+d(y,u_i)\geq 3$. It remains to note that if $d(x,u_i)+d(y,u_i)<3$, then $d(x,u_i)=d(y,u_i)=1$. Thus, $x=y=a$, and the desired inequality trivially holds.

We now show that $M$ does not have the LTP. Take $N\coloneqq \{a,b\}$. Then, for any $u,v\in M$, $u\neq v$, there exist $x,y\in N$ such that
\[
d(x,y)+d(u,v)\geq d(x,u)+d(y,v)+1.
\]
Indeed, set $U\coloneqq \{u_i\colon i\in \mathbb{N}\}$ and $V\coloneqq\{v_i\colon i\in \mathbb{N}\}$, and suppose that $u,v\in M$, $u\neq v$.
If $u,v\in U$ or $u,v\in V$, then, for $x=a$, $y=b$,
\begin{align*}
d(x,y)+d(u,v)= 4\geq 3=d(x,u)+d(y,v).
\end{align*}
If $u\in U$ and $v\in V$, or $u\in V$ and $v\in U$, then, respectively, for $x=a$, $y=b$, and for $x=b$, $y=a$,
\begin{align*}
d(x,y)+d(u,v)\geq 3>2=d(x,u)+d(y,v).
\end{align*}
Finally, if $u\in N$ or $v\in N$, then, respectively, for $x=u$, $y\in N\setminus \{x\}$, and for $y=v$, $x\in N\setminus\{y\}$,
\begin{align*}
d(x,y)+d(u,v)\geq 3>2\geq d(x,u)+d(y,v).
\end{align*}
\end{Ex}

In \cite[Proposition 4.7]{PR} it was shown that every infinite subset $M$ of $\ell_1$, viewed as a metric space, has the LTP. It turns out that every such $M$ has even the SLTP.
\begin{Ex}\label{ell1}
Every infinite subset $M$ of $\ell_1$, viewed as a metric space, has the SLTP.

Indeed, from \cite[Theorem 5.6]{HLLN} combined with our Theorem \ref{main} it follows that every unbounded metric space and every metric space $M$ with the property that $\inf\{d(x,y)\colon x,y\in M, x\neq y\}=0$ has the SLTP (this can also, without too much effort, be verified directly). Thus it suffices to consider the case when $M$ is a bounded and uniformly discrete subset of $\ell_1$. In this case there exist $R,r>0$ such that for any $x,y\in M$, $x\neq y$,
\[
r< d(x,y)< R.
\]

Let $N$ be a finite subset of $M$ and let $\varepsilon>0$. Choose $\delta >0$ such that $\varepsilon r\geq 6\delta$. Since $N$ is finite, there exists an $n\in \mathbb{N}$ such that for any $x=(x_i)\in N$
\[
\sum_{i> n} |x_i|\leq \delta.
\]
Since $M$ is infinite and bounded, there exist $u=(u_i),v=(v_i)\in M$, $u\neq v$,  such that
\[
\sum_{i\leq n} |u_i-v_i|\leq \delta.
\]
For any $x=(x_i),y=(y_i)\in N$ and $a=(a_i), b=(b_i)\in \{u,v\}$,
\begin{align*}
    \sum_i |x_i-y_i|&\leq \sum_{i\leq n} \bigr(|x_i-a_i|+|y_i-b_i|+|a_i-b_i|\bigr)+\sum_{i>n} |x_i-y_i| \\
    &\leq \sum_{i\leq n}\bigl(|x_i-a_i|+|y_i-b_i|\bigr)+3\delta
\end{align*}
and
\begin{align*}
    \sum_i |u_i-v_i|&\leq \sum_{i> n} |u_i-v_i-x_i+y_i|+\sum_{i> n}|x_i-y_i|+\sum_{i\leq n}|u_i-v_i|\\
    &\leq \sum_{i> n}\bigl(|x_i-u_i|+|y_i-v_i|\bigr)+3\delta.
\end{align*}
Therefore, for any $x=(x_i), y=(y_i), z=(z_i), w=(w_i)\in N$
\begin{align*}
(1-\varepsilon)(d(x,y)+d(u,v))&\leq d(x,y)+d(u,v)-6\delta\\
&=\sum_{i} |x_i-y_i|+\sum_i|u_i-v_i|-6\delta\\
&\leq  \sum_{i\leq n} \bigr(|x_i-u_i|+|y_i-v_i|\bigr) +3\delta\\
&\qquad +\sum_{i>n}\bigr(|x_i-u_i|+|y_i-v_i|\bigr)+3\delta-6\delta\\
&=\sum_i\bigr(|x_i-u_i|+|y_i-v_i|\bigr)\\
&=d(x,u)+d(y,v)
\end{align*}
and
\begin{align*}
    (1-&\varepsilon)\bigl(2d(u,v)+d(x,y)+d(z,w)\bigr)\\
    &\leq 2d(u,v)+d(x,y)+d(z,w)-12\delta\\
    &=2\sum_i|u_i-v_i|+ \sum_{i} |x_i-y_i|+\sum_i|z_i-w_i|-12\delta\\
    &\leq \sum_{i>n}\bigr(|x_i-u_i|+|z_i-v_i|+|y_i-u_i|+|w_i-v_i|\bigr)+6\delta\\
    &\qquad+\sum_{i\leq n}\bigr(|x_i-u_i|+|y_i-u_i|+|z_i-v_i|+|w_i-v_i|\bigr)+6\delta-12\delta\\
    &=\sum_i\bigl(|x_i-u_i|+|y_i-u_i|+|z_i-v_i|+|w_i-v_i|\bigr)\\
    &= d(x,u)+d(y,u)+d(z,v)+d(w,v).
\end{align*}
\end{Ex}

\section*{Acknowledgments}
The paper is a part of a Ph.D. thesis which is being prepared
by the author at University of Tartu under the supervision of Rainis Haller and Märt Põldvere.
The author is grateful to his supervisors for their valuable help.
This research was partially supported by institutional
research funding IUT20-57 of the Estonian Ministry of Education and
Research.

\addcontentsline{toc}{section}{References}

\end{document}